\documentclass[12pt]{article}
\usepackage{amsmath,amsfonts,amssymb,amsthm}
\usepackage{tcolorbox}
\definecolor{pantone312}{HTML}{009DD1}
\usepackage{mathtools}
\usepackage{bbm}
\usepackage{graphicx}

\usepackage{tabularx}
\usepackage{color}
\definecolor{darkgreen}{rgb}{0,0.55,0}
\newcommand{\grad}{\nabla}

\newcommand{\laplace}{\Delta}

\newcommand{\N}{\mathbbm{N}}

\newcommand{\R}{\mathbbm{R}}

\newcommand{\T}{\mathbbm{T}}
\newcommand{\thetak}{\theta^{\kappa}}

\usepackage{sectsty}

\sectionfont{\large}
\subsectionfont{\normalsize}
\subsubsectionfont{\normalsize}
\paragraphfont{\normalsize}

\newcommand{\dd}{{\rm d}}
\newcommand{\dx}{\,\dd x}

\newcommand{\dt}{\, \dd t}

\def\XXint#1#2#3{{\setbox0=\hbox{$#1{#2#3}{\int}$ }
\vcenter{\hbox{$#2#3$ }}\kern-.59\wd0}}

\newtheorem{theorem}{Theorem}
\newtheorem{lemma}{Lemma}

\pagecolor{white}

\begin{document}
\phantom{ }
\vspace{4em}

\begin{flushleft}
{\large \bf Bounds on the rate of enhanced dissipation}\\[2em]
{\normalsize \bf Christian Seis}\\[1em]

{\small Institut f\"ur Analysis und Numerik,  Westf\"alische Wilhelms-Universit\"at M\"unster,\\
Orl\'eans-Ring 10, 48149 M\"unster, Germany.\\
\emph{E-mail address}: \texttt{seis@wwu.de}}\\[2em]

{\bf Abstract:} We are concerned with flow enhanced mixing of passive scalars in the presence of diffusion. Under  the assumption that the velocity gradient is suitably integrable, we provide upper bounds on the exponential rates of  enhanced dissipation. Recent constructions indicate the optimality of   our results.

\vspace{2em}

	{\bf Statements and Declarations:} The  author states that there is no conflict of interest. Data sharing is not applicable to this article as no data sets were generated or analyzed during the current study.

\end{flushleft}

\vspace{2em}

\section{Introduction}
The theory of fluid mixing has become an active area of research in the applied mathematics community in recent years. 
Mixing refers to the homogenization process of a heterogeneously distributed physical quantity and can be driven by diffusion or the result of  advection by a straining fluid flow. Each of these transport mechanisms has a different action on the mixture. While diffusion balances local differences in concentration, which results thus in a decay in the concentration \emph{intensity}, advection creates  finer and finer filaments and acts thus on the \emph{scale} of fluctuations. The creation of small scales in turn amplifies the effect of diffusion in the fluid, a phenomenon usually referred to as \emph{enhanced dissipation}.

In order to describe these phenomena more accurately, we introduce the underlying model equation. 
Flow enhanced mixing processes of a diffusive medium in an incompressible fluid can be described by  the  advection-diffusion equation
\begin{equation}
\label{1}
\partial_t \thetak + u\cdot \grad\thetak = \kappa \laplace \thetak,
\end{equation}
that we consider, for simplicity, in the periodic box $\T^d = [0,1]^d$. Here $\thetak$ is the physical quantity of interest and $u$ is the  divergence-free velocity of the fluid. Since \eqref{1} is conservative, it is enough to consider the case in which $\thetak$ has zero mean. The  constant $\kappa$ is the diffusivity and can be interpreted as the inverse of the P\'eclet number in the non-dimensionalized setting. We shall always suppose that the diffusivity is small but finite, $\kappa\ll1$.

The above equation is linear as it is assumed that the observed quantity has no feedback on the fluid flow itself --- think, for instance, of dye in water. Such quantities are in the literature often referred to as passive scalars.

It is a well-known fact that the advection field has no impact on the $L^2$ energy balance law,
\begin{equation}\label{101}
\frac12 \frac{\dd}{\dd t}\|\theta^{\kappa}\|^2_{L^2} + \kappa \|\grad \theta^{\kappa}\|_{L^2}^2 = 0.
\end{equation}
By applying the standard  Poincar\'e estimate $\|\theta\|_{L^2}^2 \le C \|\grad \theta\|_{L^2}^2$ for mean-free functions,   we deduce a first estimate on the dissipation rate,
\begin{equation}\label{102}
\|\thetak(t)\|_{L^2} \le e^{-C\kappa t} \|\theta_0\|_{L^2},
\end{equation}
where $\theta_0$ is the initial configuration. This estimate is apparently independent of the particular  choice of the divergence-free velocity field $u$ and is optimal for the purely diffusive heat equation. However, in situations in which the fluid motion creates fine filaments,  the concentration gradients have to  increase and in view of the balance law \eqref{101},  we expect that the energy dissipates at a much higher rate. To be more specific and to fix terminology, we are concerned with the phenomenon of    \emph{dissipation enhancement}  if there exists a constant $D(\kappa)\gg \kappa$ such that
\begin{equation}
\label{302}
\|\thetak(t)\|_{L^2} \le e^{-D(\kappa) t} \|\theta_0\|_{L^2}
\end{equation}
for every choice of the initial datum.

A first qualitative evidence  of such enhanced dissipation effects was provided in \cite{ConstantinKiselevRyzhikZlatos08}. In this paper, a sharp characterization of (steady) incompressible flows that are dissipation enhancing is established. Quantitative results with precise exponential decay rates have been obtained only very recently. Particularly well understood is the effect of enhanced dissipation in the class of shear flows. Here, the dissipation rate increases to $D(\kappa) \sim \kappa^{\gamma}$ for some $\gamma\in(0,1)$ as observed, for instance, in \cite{BedrossianCotiZelati17,Wei19,CotiZelatiDrivas19,	CotiZelati20,ColomboCotiZelatiWidmayer20,
CotiZelatiDolce20}.\footnote{Here and in the sequel, we write $A\lesssim B$ if there exists a uniform constant $C$ independent of $\kappa$ such that $A\le CB$. Moreover, we write $A \sim B$ if $A\lesssim B $ and $B\lesssim A$. Finally $A\ll B$ if $A\le C B$ for some sufficiently  small $C$.} The results in these works rely on the particularly simple structure of shear velocity fields and there is no (obvious) way of translating the techniques developed in there to chaotic or turbulent fluid motions. The only rigorous result known to the author in which enhanced dissipation could be established in a more complex setting is for randomly forced  (and, in fact, chaotic \cite{BedrossianBlumenthalPunshon-Smith18}) fluids systems, which include stochastic Stokes and 2D Navier--Stokes equations \cite{BedrossianBlumenthalPunshonSmith19a}: It features a dissipation rate that depends logarithmically on the diffusivity constant, $D(\kappa) \sim \log^{-1}\frac1{\kappa}$.
 The velocity fields constructed in this work are quite regular as they satisfy  stochastic analogues  of the bound
\begin{equation}\label{303}
\int_0^t \|\grad u\|_{L^2}\dt \lesssim 1+t.
\end{equation}

In the present work, we make an attempt to derive  bounds  on the maximal rates of enhanced dissipation that apply to a large class of velocity fields and require only assumptions on their regularity but not on their structure. For instance, we will show that \emph{for any flow} in the regularity class \eqref{303}, the rate of enhanced dissipation is bounded as $D(\kappa) \lesssim \log^{-1}\frac1{\kappa}$, and thus, if we neglect the stochastic origin for a moment, our findings prove that the estimates in  \cite{BedrossianBlumenthalPunshonSmith19a} are optimal. Vice versa, as the construction in the latter work saturates our bound, the estimate $D(\kappa) \lesssim \log^{-1}\frac1{\kappa}$ derived here \emph{has to be sharp}.

The rates of enhanced dissipation are intimately related  to the rates of mixing (i.e., the decay of scales) in the non-diffusive setting, usually measured in terms of negative Sobolev norms \cite{MathewMezicPetzold05,LinThiffeaultDoering11,Thiffeault12}. A first rigorous connection between both phenomena was established in  \cite{FengIyer19,CotiZelatiDelgadinoElgindi20}. In \cite{CotiZelatiDelgadinoElgindi20}, the authors obtain a rate $D(\kappa) \sim \log^{-2}\frac1{\kappa}$, provided that there exists a Lipschitz regular velocity field that mixes (for $\kappa=0$) any initial datum at an exponential rate. Our result thus shows the optimality possibly modulo a power on the logarithm.  
 Exponential rates are optimal in the non-diffusive setting for velocity fields in the class \eqref{303}, see \cite{Depauw03,CrippaDeLellis08,Lunasin_etal12,Seis13,IyerKiselevXu14,
Leger18,YaoZlatos17,AlbertiCrippaMazzucato19,ElgindiZlatos19,
BedrossianBlumenthalPunshonSmith19b}.

We remark that proving upper bounds on the rates of enhanced dissipation is quite different from proving lower bounds on the $L^2$ decay. Indeed, the true counterpart of \eqref{302} would be a lower bound of the form
\[
\|\thetak(t)\|_{L^2} \ge e^{-D(\kappa)t} \|\theta_0\|_{L^2},
\]
which is a long-standing open problem (and is, in fact, controversial). Moreover, it is clear that bounds of this type cannot be true without additional  assumptions on the regularity of the initial datum since they are false for the heat equation. Of course, the same problem applies to the question that we address here.

In this regard, the progress made in the present paper is certainly modest. Yet, our contribution is substantial as it establishes  limitations on the effect of enhanced dissipation for the first time. Moreover, upper bounds  on the dissipation rate $D(\kappa)$ 
single out a time scale that is characteristic for the  diffusion in the presence of an irregular flow field.

We shall now state and discuss our precise result.

\begin{theorem}\label{T1}
Suppose that $u$ is a divergence-free velocity field  satisfying
\begin{equation}\label{304}
\int_0^t \|\grad u\|_{L^p}\dt \lesssim 1 + t^{\alpha},
\end{equation}
 for some $p\in(1,\infty]$ and $\alpha\in [0,1]$.
  Let $q\in[1,\infty]$ be such that $\frac1p+\frac1q\le 1$, and suppose  $\theta_0$ is a mean-zero initial configuration  satisfying $\|\theta_0\|_{L^1}\sim \|\theta_0\|_{L^q}\sim 1$ and $\|\grad\theta_0\|_{L^1}\lesssim 1$. If there are   positive  constants $D$ and $\Lambda$ such that
\begin{equation}\label{2}
 \|\thetak(t)\|_{L^q}   \le \Lambda e^{-Dt}\|\theta_0\|_{L^q} ,
\end{equation}
for any $t>0$, then there exists a constant $C>1$ independent of $D,\Lambda$ and $\kappa$ such that
\begin{equation}
\label{22}
D\lesssim \begin{cases}\Lambda^{\frac1{\alpha}}\log^{-\frac1{\alpha}}\frac1{\kappa}& \mbox{if }\alpha>0,\\
e^{C\Lambda}\kappa &\mbox{if }\alpha=0,
\end{cases}
\end{equation}
provided that $\kappa \ll1$.
\end{theorem}

For notational simplicity, we will focus on the case $p=q=2$ in the following discussion of Theorem \ref{T1}. We first note that the velocity field in \eqref{303} corresponds to the case $\alpha=1$ in \eqref{304} and generalizes the important class of velocity fields with a fixed power budget $\|\grad u\|_{L^2} \le 1$. In industrial stirring processes, this quantity describes  the amount of work an agent spends per time unit to maintain stirring. We have included the general factor $\Lambda$ in \eqref{2} to gain some degree of freedom. For instance, if we consider enhanced dissipation estimates in the form \eqref{302}, this $\Lambda$ is just $1$. This way, we deduce $D\lesssim \log^{-1}\frac1{\kappa}$ from \eqref{22} as announced earlier, and this estimate is optimal in light of the results from \cite{BedrossianBlumenthalPunshonSmith19a}. On the other hand, when studying enhanced dissipation with diffusivity independent rates $D\sim 1$, our findings indicate that the prefactor $\Lambda$ has to diverge at least logarithmically, $\log\frac1{\kappa}\lesssim \Lambda$. Whether this estimate is optimal or not is not known to the author. We remark that the work \cite{BedrossianBlumenthalPunshonSmith19a} provides an estimate with a diffusivity-independent rate, $D\sim 1$, and a diverging prefactor of size $\Lambda\sim \frac1{\kappa}$. To the best of our knowledge, diffusivity-independent rates have not (yet) been observed in the deterministic setting. Apparently, they do not result from a spectral gap estimate and it is not clear to the author whether these are generic.

In a certain sense, the two different considerations $\Lambda=1$ and $D\sim 1$ correspond to global-in-time and large-time estimates, respectively. Indeed, setting $t=0$ in \eqref{2} and considering \eqref{102}, it is obvious that  $\Lambda=1$ applies at small times and thus this choice  in \eqref{2} gives a globally-in-time applicable dissipation rate. To understand why $D\sim1$ is expected for the large time dynamics, it is worth to have a short but oversimplified heuristic discussion on flow-induced mixing in diffusive media. In case that the initial configuration is sufficiently smooth, early stage mixing is essentially due to advection, that is, the reduction of the average scale. In this stage, the decay of the $L^2$ norm is rather slow. Diffusion becomes relevant only at later times, when the typical scale is reduced to  the so-called Batchelor scale \cite{Batchelor59}, at which diffusion and advection balance. Dimensional  arguments suggest that this scale is of the order $\sqrt{\kappa}$, independently of the precise value of $\alpha>0$ in our assumption \eqref{304}.  The later time dynamics are then governed by the diffusion process, with an exponential rate $D$ proportional to $\kappa |k|^2$, where $k$ is the  dominated wavenumber, which should be inverse to the Batchelor scale. Hence, $D\sim \kappa |k|^2 \sim 1$ is the expected dissipation rate for large times. It was observed numerically in \cite{MilesDoering18}. The prefactor $\Lambda$ in this case is difficult to predict because it depends sensitively on the precise rate of early stage mixing and on the crossover time (which is decreasing with $\alpha$) from advection-dominated to diffusion-dominated mixing. Our estimates in \eqref{22} provide first rigorous bounds on that number.

The cases $\alpha\in(0,1)$ correspond to situations in which the mixing efficiency of the flow is decaying in time, which is relevant, for instance, in case of the Navier--Stokes equations without forcing. Here,  the effect of viscous dissipation slows down the velocity field and the  bound \eqref{304} holds true with $\alpha=1/2$.  A more academic application is a  mixing process with given, but decaying power budget. 
Our estimates suggest that the dissipation rate  $D$ in \eqref{302} is getting smaller when the advection  slows down. This is certainly expected, and our bound is again consistent with the findings in \cite{CotiZelatiDelgadinoElgindi20} modulo a factor of $2$ in the exponent. In case the mixing flow becomes negligible in finite time, $\alpha=0$, we cannot expect more than a change in the prefactor compared to the purely diffusive bound \eqref{102}, and thus, in this case \eqref{22} is sharp. If one aims for dissipation rates that are uniform in $D$, our estimates imply that $\Lambda \gtrsim \log\frac1\kappa$ for any value of $\alpha$. This lower bound is very likely not optimal because $\Lambda$ is supposed to be growing as a function of  $\alpha$. The author is not aware of any rigorous bounds on $\Lambda$ apart from those presented here.

We finally remark that the bound in \eqref{22} can be chosen independently of the precise gradient bound on the initial datum as long as $1/\kappa$ is large compared to $\|\grad \theta_0\|_{L^1}$, which is guaranteed in the hypothesis of the theorem.


Around the same time a first version of the  present paper was distributed by the author, Bru\`e and Nguyen uploaded a paper on the arXiv, that contains (among other results on diffusive mixing) estimates that are similar to (but slightly weaker than)  those in Theorem \ref{T1}, cf.~\cite{BrueNguyen20}. Indeed, in this work, the case $\alpha=1$ , $p>2 $ and $\Lambda\sim 1$ is treated and $D$ is bounded by $\log^{-\frac{p-1}{p}}\frac1{\kappa}$.

Considering the fact that estimates on enhanced dissipation have been completely open for many years, the proof of Theorem \ref{T1} is surprisingly short. It relies, however, heavily on a stability theory for continuity equations developed by the author in \cite{Seis17}, which in turn grew out of the Crippa--De Lellis theory for Lagrangian flows \cite{CrippaDeLellis08} and its Eulerian adaption in  studies of coarsening problems  \cite{BrenierOttoSeis11,OttoSeisSlepcev13} and mixing phenomena \cite{Seis13}. An understanding of the role of diffusive perturbations, which is, in a certain sense, the view point taken here, has been developed in the study of numerical schemes featuring numerical diffusion \cite{SchlichtingSeis17,SchlichtingSeis18}. The theory is based on suitable Kantorovich--Rubinstein distances, which have their origin in the theory of optimal mass transportation.
For a review of our method, we refer to \cite{Seis18}.

The remainder of this article is devoted to the proofs.

%
%
%
%
%
%
%
%

\section{Proofs}

We will make use of the following Kantorovich--Rubinstein distance with logarithmic cost function, which was introduced in \cite{Seis17}. For $\delta>0$ and any mean zero function $\theta$ on $\T^d$, we define
\[
D_{\delta}(\theta) = \inf_{\pi \in \Pi(\theta^+,\theta^-)} \iint \log\left(\frac{|x-y|}{\delta} +1\right) \dd\pi(x,y),
\]
where $\theta^+$ and $\theta^-$ denote the positive and negative parts of $\theta$, respectively, and $\Pi(\theta^+,\theta^-)$ is the set of transport plans $\pi :\T^d\times \T^d\to \R_+$ with marginals $\theta^+$ and $\theta^-$, i.e.,
\[
\iint \varphi(x)+\psi(y)\, \dd\pi(x,y)  = \int \varphi \theta^+\dx + \int \psi \theta^- \dx,
\]
for all continuous functions $\varphi$ and $\psi$ on the torus. We remark that the Kantorovich--Rubinstein distance is finite only if $\theta$ has zero mean, because then, both $\theta^+$ and $\theta^-$ have the same total mass,
\[
\int \theta^+\dx = \int \theta^-\dx.
\]

Our subsequent proofs will not use many properties of Kantorovich--Rubinstein distances, as some of the key estimates, above all the following lemma, can be taken from the existing literature. Yet, we refer the interested reader  to \cite{Villani03} for a comprehensive introduction into the theory of optimal transportation.

The rate of change of $D_{\delta}$   under solutions to  advection-diffusion equations has been investigated in \cite{OttoSeisSlepcev13,Seis18}, see also \cite{BrenierOttoSeis11,Seis13,Seis17} for  related estimates in the purely advective case. 

\begin{lemma}[\cite{OttoSeisSlepcev13,Seis18}]
\label{L1}
Let $\thetak$ be a mean-zero solution to the advection-diffusion equation \eqref{1}. Then $D_{\delta}(\thetak)$ is absolutely continuous and it holds
\begin{equation}
\label{4}
\left|\frac{\dd}{\dd t} D_{\delta}(\thetak)\right| \lesssim  \|\grad u\|_{L^p}  \|\thetak\|_{L^{p'}}  +\frac{\kappa}{\delta} \|\grad\thetak\|_{L^1},
\end{equation}
where $\frac1p+\frac1{p'}=1$. 
\end{lemma}

Apart from its applications for mixing that we elaborate in the following, this estimate can be used to quantify the  (weak) convergence in the vanishing diffusivity  limit $\kappa \to 0$.  In this context, $\delta$ can be interpreted as the order of convergence. This observation has been exploited in order to bound the approximation error due to numerical diffusion generated by the upwind finite volume scheme for continuity equations in \cite{SchlichtingSeis17,SchlichtingSeis18}.

Our next result is a lower bound on the Kantorovich--Rubinstein distance in terms of the $L^1$ norms of $\theta$ and its gradient.

\begin{lemma}
\label{L3}
Let $\theta$ be a mean zero function in $W^{1,1}(\T^d)$. Then there exists a constant $C>0$ such that  
\begin{equation}
\label{7}
D_{\delta}(\theta) \gtrsim \log \left(\frac{\|\theta\|_{L^1}}{\delta C\|\grad\theta\|_{L^1}}+1\right) \|\theta\|_{L^1}.
\end{equation}
\end{lemma}
The statement of the lemma is a consequence of an interpolation inequality between Kantorovich--Rubinstein distances with logarithmic cost function and the Sobolev norm, a variation of which was proved previously in \cite{BrenierOttoSeis11,OttoSeisSlepcev13,Seis13}. It is a generalization of the  endpoint Kantorovich--Sobolev inequality
\[
1\lesssim \log^{-1}\left(\|\grad\theta\|_{L^1}^{-1}+1\right)\inf_{\pi\in\Pi(\theta^+,\theta^-)} \iint \log\left(|x-y|+1\right)\, \dd \pi(x,y) 
\]
for probability distributions, see~\cite{Ledoux15} for standard Wasserstein versions.

\begin{proof}
We recall that by duality, it holds that
\begin{equation}\label{17}
 \|\theta\|_{L^1} = \sup_{\|\psi\|_{L^{\infty}}\le 1}\int \theta\psi\dx.
\end{equation}
We now pick $\psi$ arbitrary with $\|\psi\|_{L^{\infty}}\le 1$  and denote by subscript $R$ the convolution with a standard mollifier of scale $R$. We then split
\begin{equation}\label{13}
\int \theta\psi \dx = \int(\theta-\theta_R)\psi\dx + \int \theta \psi_R\dx,
\end{equation}
where we have used  symmetry properties of the mollifier to shift the subscript from $\theta$ to $\psi$. 

For the first term, we use the fact that $\|\theta-\theta_R\|_{L^1} \lesssim R\|\grad \theta\|_{L^1}$, so that
\begin{equation}\label{14}
\int (\theta-\theta_R)\psi\dx  \lesssim R\|\grad\theta\|_{L^1}.
\end{equation}
For the second one, we introduce a second auxiliary length scale $r$ and write
\begin{align*}
\MoveEqLeft\int \theta\psi_R\dx\\
&= \int (\theta^+-\theta^-)\psi_R\dx\\
& = \iint (\psi_R(x)-\psi_R(y))\,\dd\pi(x,y)\\
& = \iint_{|x-y|\le r} (\psi_R(x)-\psi_R(y))\,\dd\pi(x,y)+ \iint_{|x-y|> r} (\psi_R(x)-\psi_R(y))\,\dd\pi(x,y),
\end{align*}
where $\pi\in \Pi(\theta^+,\theta^-)$ is an arbitrary transport plan and we have used its marginal conditions in the second equality. On the one hand, because $\psi_R$ is Lipschitz and $\|\grad\psi_R\|_{L^{\infty}}\lesssim \frac1R\|\psi\|_{L^{\infty}}\le \frac1R$, we have that
\begin{align*}
 \iint_{|x-y|\le r} (\psi_R(x)-\psi_R(y))\,\dd\pi(x,y) \lesssim r \|\grad\psi_R\|_{L^{\infty}}\iint   \dd\pi(x,y) \lesssim \frac{r}R \|\theta\|_{L^1}.
\end{align*}
On the other hand, using $\|\psi_R\|_{L^{\infty}}\le \|\psi\|_{L^{\infty}}\le 1$, the monotonicity of the logarithm and setting $c(z) = \log\left(\frac{z}{\delta}+1\right)$, we estimate
\begin{align*}
\iint_{|x-y|> r} (\psi_R(x)-\psi_R(y))\,\dd\pi(x,y) & \le \frac{2\|\psi_R\|_{L^{\infty}}}{c(r)} \iint c(|x-y|)\, \dd\pi(x,y)\\
&\lesssim \frac{1}{c(r)} \iint c(|x-y|)\, \dd\pi(x,y).
\end{align*}
Combining the previous estimates and optimizing in $\pi$ on the right-hand side, we conclude that 
\[
\int\theta\psi_R\dx \lesssim \frac{r}{R}\|\theta\|_{L^1} + \frac1{c(r)}D_{\delta}(\theta).
\]
Plugging this estimate and \eqref{14} into the decomposition \eqref{13}, we arrive at
\[
\int \theta\psi\dx \lesssim R\|\grad\theta\|_{L^1} + \frac{r}{R}\|\theta\|_{L^1} + \frac1{c(r)}D_{\delta}(\theta),
\]
for any $\psi$ such that $\|\psi\|_{L^{\infty}}\le 1$. Maximizing in $\psi$ on the left-hand side and choosing $R\gg r$, we deduce that
\[
\|\theta\|_{L^1} \lesssim r \|\grad\theta\|_{L^1} + \frac1{c(r)}D_{\delta}(\theta),
\]
and thus, the result follows upon choosing $r\ll \frac{\|\theta\|_{L^1}}{\|\grad\theta\|_{L^1}}$.
\end{proof}

We are now in the position to prove our bound on the dissipation rate. 

\begin{proof}[Proof of Theorem \ref{T1}] We may without loss of generality assume that $D\le 1$ and $\Lambda\ge 1$. 

We notice that $1/D$ is the dissipation time scale, and we  set $t_n = n/D$ for $n\in\N$.
Integrating \eqref{4} over $[t_n,t_{n+1}]$  yields
\begin{equation}\label{331}
\left|D_{\delta}(\thetak(t_{n+1})) - D_{\delta}(\thetak(t_n))\right| \lesssim \int_{t_n}^{t_{n+1}}\|\grad u\|_{L^p} \|\thetak\|_{L^q}\dt +\frac{\kappa}{\delta} \int_{t_n}^{t_{n+1}} \|\grad\thetak\|_{L^1}\, dt.
\end{equation}

If $q\ge 2$, we use Jensen's inequality and the energy balance \eqref{101} to bound the gradient term on the right-hand side,
\begin{align*}
\frac{\kappa}{\delta} \int_{t_n}^{t_{n+1}} \|\grad\thetak\|_{L^1}\, dt &\le\frac1{\delta} \sqrt{\frac{\kappa}D}   \left(\kappa \int_{t_n}^{t_{n+1}} \|\grad \thetak\|_{L^2}^2\dt\right)^{\frac12} \\
&\le \frac1{\delta} \sqrt{\frac{\kappa}D}    \|\thetak(t_n)\|_{L^2}\\
&\le \frac1{\delta} \sqrt{\frac{\kappa}D}    \|\thetak(t_n)\|_{L^q}.
\end{align*}
Otherwise, if $q\le 2$, we use the generalized energy equality
\[
\|\thetak(t_{n+1})\|_{L^q}^q + \kappa q(q-1)\int_{t_n}^{t_{n+1}} \int_{\T^d} |\thetak|^{q-2}|\grad\thetak|^2\dx\dt  = \|\thetak(t_n)\|^q_{L^q},
\]
which is derived via a standard computation, and we estimate via interpolation and Jensen's inequality
\begin{align*}
\frac{\kappa}{\delta} \int_{t_n}^{t_{n+1}} \|\grad\thetak\|_{L^1}\dt & \le \frac{\kappa}{\delta} \int_{t_n}^{t_{n+1}} \left(\int_{\T^d} |\thetak|^{q-2}|\grad\thetak|^2\dx\right)^{\frac12}\|\thetak\|_{L^{2-q}}^{\frac{2-q}2}\dt\\
&\le\frac1{\delta} \sqrt{\frac{\kappa}D}
\left(\kappa \int_{t_n}^{t_{n+1}} \int_{\T^d} |\thetak|^{q-2}|\grad\thetak|^2\dx\dt\right)^{\frac12} \|\thetak(t_n)\|_{L^q}^{\frac{2-q}q}\\
&\le \frac1{\delta} \sqrt{\frac{\kappa}D}  \|\thetak(t_n)\|_{L^q}.
\end{align*}
In either case, we have that 
\begin{equation}
\label{305}
\frac{\kappa}{\delta} \int_{t_n}^{t_{n+1}} \|\grad\thetak\|_{L^1}\dt \lesssim\frac1{\delta} \sqrt{\frac{\kappa}D}   \|\thetak(t_n)\|_{L^q} \le   \frac{\Lambda}{\delta} \sqrt{\frac{\kappa}D}    e^{-Dt_n},
\end{equation}
where we have used \eqref{2} in the last inequality.

Regarding the velocity term in \eqref{331}, we now observe that the budget constraint  \eqref{304} and the enhanced dissipation assumption \eqref{2} imply that
\begin{equation}\label{330}
\begin{aligned}
\int_{t_n}^{t_{n+1}} \|\grad u\|_{L^p}\|\thetak\|_{L^q}\dt & \lesssim \Lambda e^{-Dt_n} \left(1+t_{n+1}^{\alpha}\right)\\
&\lesssim \Lambda e^{-Dt_{n+2}} t_{n+2}^{\alpha}\\
&\lesssim \Lambda  D^{-\alpha} e^{-Dt_{n+2}/2} ,
\end{aligned}
\end{equation}
where in the last inequality we use that the mapping $t\mapsto t^{\alpha} e^{-Dt/2}$ is decreasing for $t\ge   t_2 $.

Inserting the two estimates \eqref{305} and \eqref{330} into \eqref{331} then gives the bound
\[
\left|D_{\delta}(\thetak(t_{n+1})) - D_{\delta}(\thetak(t_n))\right| \lesssim \Lambda D^{-\alpha} e^{-Dt_{n+2}/2} +\frac{\Lambda}{\delta} \sqrt{\frac{\kappa}D}    e^{-Dt_n}.
\]
Hence, summing over $n$ and recalling that $t_n = n/D$, we find that
\begin{equation}\label{400}
\begin{aligned}
D_{\delta}(\theta_0) 
&\lesssim D_{\delta}(\thetak(t_N)) + \frac{\Lambda}{D^{\alpha}}  \sum_{n=0}^{N-1} e^{-\frac{n}2}   +\frac{\Lambda}{\delta}\sqrt{\frac{\kappa}D} \sum_{n=0}^{N-1}e^{-n} \\
&\lesssim D_{\delta}(\thetak(t_N)) + \frac{\Lambda}{D^{\alpha}}      +\frac{\Lambda}{\delta}\sqrt{\frac{\kappa}D} .
\end{aligned}
\end{equation}

We now use the lower bound on our Kantorovich--Rubinstein distance \eqref{7} and the assumption on the initial datum   to estimate the left-hand side from below. It holds that
\[
\log\left(\frac1{C\delta}+1\right) \lesssim D_{\delta}(\theta_0),
\]
for some $C>0$. This constant can be chosen larger than $1$ without restrictions, and thus,
\[
\log\left(\frac1{C\delta}+1\right)\ge  \log \left(\frac1{\delta}+1\right)-\log C\gtrsim \log \left(\frac1{\delta}+1\right)\ge \log\frac1{\delta},
\]
if $\delta$ is sufficiently small, which we will ensure later. On the other hand,   because $|x-y|\le 1$ on the torus, we  have the following brutal estimate on the Kantorovich--Rubinstein distance
\[
D_{\delta}(\thetak) \le \log\left(\frac1{\delta}+1\right) \iint \dd\pi(x,y) \lesssim \log\left(\frac1{\delta}+1\right) \|\thetak\|_{L^1},
\]
and thus by \eqref{2}, and the fact that $Dt_N=N$, \eqref{400} becomes
\[
\log \frac1{\delta} \lesssim\Lambda \log\left(\frac1{\delta}+1\right) e^{-N} + \frac{\Lambda}{D^{\alpha}}     +\frac{\Lambda}{\delta}\sqrt{\frac{\kappa}D}.
\]
Because $N$ was arbitrary, the first term on the right-hand side can be dropped. We thus arrive at
\begin{equation}\label{200}
\frac1C\log\frac1{\delta} \le  \frac{\Lambda}{D^{\alpha}}  +\frac{\Lambda}{\delta}\sqrt{\frac{\kappa}D},
\end{equation}
for some constant $C$ independent of $\Lambda$ and $D$.

To conclude, we consider separately the cases $\alpha\le 1/2$ and $\alpha>1/2$. 

\emph{Case $\alpha\le 1/2$.} We may without loss of generality assume that $\Lambda^2 \kappa \ll D$, because otherwise, if $D\lesssim \Lambda^2 \kappa$, the statement can be deduced from the estimate
\[
\Lambda^2 \kappa \lesssim \Lambda^{\frac1{\alpha}}\log^{-\frac1{\alpha}} \frac1{\kappa},
\]
which is trivially satisfied since $\Lambda\ge 1$ and $\kappa\ll1$. We optimize \eqref{200} with respect to $\delta$ by choosing $\delta = C\Lambda \sqrt{\frac{\kappa}D}$. Note that this $\delta$ can  assumed to be small by the previous argument.  With this,  estimate \eqref{200} becomes
\[
0 \le \frac{C \Lambda}{D^{\alpha}} + 1 + \log\left(C\Lambda\sqrt{\frac{\kappa}D}\right).
\]
On the one hand, if $\alpha=0$, the latter implies that
\[
\log\frac{D}{\kappa} \lesssim 1+ \log\Lambda + \Lambda \lesssim \Lambda,
\]
since $\Lambda\ge 1$. On the other hand, if $\alpha>0$, we deduce   for small $\kappa$ that
\[
\log\frac1{\kappa} \lesssim \frac{\Lambda}{D^{\alpha}}+ \log\Lambda  +\log\frac1D\lesssim \frac{\Lambda}{D^{\alpha}}.
\]
In either case, we infer the  desired estimate.

\emph{Case $\alpha>1/2$.} In this case, we choose $\delta  = \kappa^{\frac12} D^{\alpha- \frac12}$ in \eqref{200}, which is small  because $D\le1 $, to the effect that
\[
\log\frac1{\kappa} \lesssim \log\frac1{\kappa} + \log\frac1D \lesssim \frac{\Lambda}{D^{\alpha}}.
\]
Again, this is the stated estimate.

\end{proof}

\section*{Acknowledgement} 
This work is funded by the Deutsche Forschungsgemeinschaft (DFG, German Research Foundation) under Germany's Excellence Strategy EXC 2044 --390685587, Mathematics M\"unster: Dynamics--Geometry--Structure. The author thanks Hung Nguyen for sharing his preprint \cite{BrueNguyen20} and  Jacob Bedrossian for enlightening discussions on the results of \cite{BedrossianBlumenthalPunshonSmith19a}. He acknowledges critical comments by the anoymous referee.

\bibliography{euler_richtig}

\begin{thebibliography}{10}

\bibitem{AlbertiCrippaMazzucato19}
{\sc Alberti, G., Crippa, G., and Mazzucato, A.~L.}
\newblock Exponential self-similar mixing by incompressible flows.
\newblock {\em J. Amer. Math. Soc. 32}, 2 (2019), 445--490.

\bibitem{Batchelor59}
{\sc Batchelor, G.~K.}
\newblock Small-scale variation of convected quantities like temperature in
  turbulent fluid. {I}. {G}eneral discussion and the case of small
  conductivity.
\newblock {\em J. Fluid Mech. 5\/} (1959), 113--133.

\bibitem{BedrossianBlumenthalPunshon-Smith18}
{\sc Bedrossian, J., Blumenthal, A., and Punshon-Smith, S.}
\newblock Lagrangian chaos and scalar advection in stochastic fluid mechanics,
  2018.

\bibitem{BedrossianBlumenthalPunshonSmith19a}
{\sc Bedrossian, J., Blumenthal, A., and Punshon-Smith, S.}
\newblock Almost-sure enhanced dissipation and uniform-in-diffusivity
  exponential mixing for advection-diffusion by stochastic navier-stokes, 2019.

\bibitem{BedrossianBlumenthalPunshonSmith19b}
{\sc Bedrossian, J., Blumenthal, A., and Punshon-Smith, S.}
\newblock Almost-sure exponential mixing of passive scalars by the stochastic
  navier-stokes equations, 2019.

\bibitem{BedrossianCotiZelati17}
{\sc Bedrossian, J., and Coti~Zelati, M.}
\newblock Enhanced dissipation, hypoellipticity, and anomalous small noise
  inviscid limits in shear flows.
\newblock {\em Arch. Ration. Mech. Anal. 224}, 3 (2017), 1161--1204.

\bibitem{BrenierOttoSeis11}
{\sc Brenier, Y., Otto, F., and Seis, C.}
\newblock Upper bounds on coarsening rates in demixing binary viscous liquids.
\newblock {\em SIAM J. Math. Anal. 43}, 1 (2011), 114--134.

\bibitem{BrueNguyen20}
{\sc Bruè, E., and Nguyen, Q.-H.}
\newblock Advection diffusion equations with sobolev velocity field, 2020.

\bibitem{ColomboCotiZelatiWidmayer20}
{\sc Colombo, M., Zelati, M.~C., and Widmayer, K.}
\newblock Mixing and diffusion for rough shear flows, 2020.

\bibitem{ConstantinKiselevRyzhikZlatos08}
{\sc Constantin, P., Kiselev, A., Ryzhik, L., and Zlato\v{s}, A.}
\newblock Diffusion and mixing in fluid flow.
\newblock {\em Ann. of Math. (2) 168}, 2 (2008), 643--674.

\bibitem{CotiZelati20}
{\sc Coti~Zelati, M.}
\newblock Stable mixing estimates in the infinite {P}\'{e}clet number limit.
\newblock {\em J. Funct. Anal. 279}, 4 (2020), 108562, 25.

\bibitem{CotiZelatiDelgadinoElgindi20}
{\sc {Coti Zelati}, M., Delgadino, M.~G., and Elgindi, T.~M.}
\newblock On the relation between enhanced dissipation timescales and mixing
  rates.
\newblock {\em Comm. Pure Appl. Math. (online first)\/} (2019).

\bibitem{CotiZelatiDolce20}
{\sc Coti~Zelati, M., and Dolce, M.}
\newblock Separation of time-scales in drift-diffusion equations on
  {$\Bbb{R}^2$}.
\newblock {\em J. Math. Pures Appl. (9) 142\/} (2020), 58--75.

\bibitem{CotiZelatiDrivas19}
{\sc {Coti Zelati}, M., and Drivas, T.~D.}
\newblock A stochastic approach to enhanced diffusion, 2019.

\bibitem{CrippaDeLellis08}
{\sc Crippa, G., and De~Lellis, C.}
\newblock Estimates and regularity results for the {D}i{P}erna-{L}ions flow.
\newblock {\em J. Reine Angew. Math. 616\/} (2008), 15--46.

\bibitem{Depauw03}
{\sc Depauw, N.}
\newblock Non unicit\'{e} des solutions born\'{e}es pour un champ de vecteurs
  {BV} en dehors d'un hyperplan.
\newblock {\em C. R. Math. Acad. Sci. Paris 337}, 4 (2003), 249--252.

\bibitem{ElgindiZlatos19}
{\sc Elgindi, T.~M., and Zlato\v{s}, A.}
\newblock Universal mixers in all dimensions.
\newblock {\em Adv. Math. 356\/} (2019), 106807, 33.

\bibitem{FengIyer19}
{\sc Feng, Y., and Iyer, G.}
\newblock Dissipation enhancement by mixing.
\newblock {\em Nonlinearity 32}, 5 (2019), 1810--1851.

\bibitem{IyerKiselevXu14}
{\sc Iyer, G., Kiselev, A., and Xu, X.}
\newblock Lower bounds on the mix norm of passive scalars advected by
  incompressible enstrophy-constrained flows.
\newblock {\em Nonlinearity 27}, 5 (2014), 973--985.

\bibitem{Ledoux15}
{\sc Ledoux, M.}
\newblock Sobolev-{K}antorovich inequalities.
\newblock {\em Anal. Geom. Metr. Spaces 3}, 1 (2015), 157--166.

\bibitem{Leger18}
{\sc L\'{e}ger, F.}
\newblock A new approach to bounds on mixing.
\newblock {\em Math. Models Methods Appl. Sci. 28}, 5 (2018), 829--849.

\bibitem{LinThiffeaultDoering11}
{\sc Lin, Z., Thiffeault, J.-L., and Doering, C.~R.}
\newblock Optimal stirring strategies for passive scalar mixing.
\newblock {\em J. Fluid Mech. 675\/} (2011), 465--476.

\bibitem{Lunasin_etal12}
{\sc Lunasin, E., Lin, Z., Novikov, A., Mazzucato, A., and Doering, C.~R.}
\newblock Optimal mixing and optimal stirring for fixed energy, fixed power, or
  fixed palenstrophy flows.
\newblock {\em J. Math. Phys. 53}, 11 (2012), 115611, 15.

\bibitem{MathewMezicPetzold05}
{\sc Mathew, G., Mezi\'{c}, I., and Petzold, L.}
\newblock A multiscale measure for mixing.
\newblock {\em Phys. D 211}, 1-2 (2005), 23--46.

\bibitem{MilesDoering18}
{\sc Miles, C.~J., and Doering, C.~R.}
\newblock Diffusion-limited mixing by incompressible flows.
\newblock {\em Nonlinearity 31}, 5 (2018), 2346--2350.

\bibitem{OttoSeisSlepcev13}
{\sc Otto, F., Seis, C., and Slep\v{c}ev, D.}
\newblock Crossover of the coarsening rates in demixing of binary viscous
  liquids.
\newblock {\em Commun. Math. Sci. 11}, 2 (2013), 441--464.

\bibitem{SchlichtingSeis17}
{\sc Schlichting, A., and Seis, C.}
\newblock Convergence rates for upwind schemes with rough coefficients.
\newblock {\em SIAM J. Numer. Anal. 55}, 2 (2017), 812--840.

\bibitem{SchlichtingSeis18}
{\sc Schlichting, A., and Seis, C.}
\newblock Analysis of the implicit upwind finite volume scheme with rough
  coefficients.
\newblock {\em Numer. Math. 139}, 1 (2018), 155--186.

\bibitem{Seis13}
{\sc Seis, C.}
\newblock Maximal mixing by incompressible fluid flows.
\newblock {\em Nonlinearity 26}, 12 (2013), 3279--3289.

\bibitem{Seis17}
{\sc Seis, C.}
\newblock A quantitative theory for the continuity equation.
\newblock {\em Ann. Inst. H. Poincar\'{e} Anal. Non Lin\'{e}aire 34}, 7 (2017),
  1837--1850.

\bibitem{Seis18}
{\sc Seis, C.}
\newblock Optimal stability estimates for continuity equations.
\newblock {\em Proc. Roy. Soc. Edinburgh Sect. A 148}, 6 (2018), 1279--1296.

\bibitem{Thiffeault12}
{\sc Thiffeault, J.-L.}
\newblock Using multiscale norms to quantify mixing and transport.
\newblock {\em Nonlinearity 25}, 2 (2012), R1--R44.

\bibitem{Villani03}
{\sc Villani, C.}
\newblock {\em Topics in optimal transportation}, vol.~58 of {\em Graduate
  Studies in Mathematics}.
\newblock American Mathematical Society, Providence, RI, 2003.

\bibitem{Wei19}
{\sc Wei, D.}
\newblock Diffusion and mixing in fluid flow via the resolvent estimate.
\newblock {\em Sci. China Math.\/} (2019).

\bibitem{YaoZlatos17}
{\sc Yao, Y., and Zlato\v{s}, A.}
\newblock Mixing and un-mixing by incompressible flows.
\newblock {\em J. Eur. Math. Soc. (JEMS) 19}, 7 (2017), 1911--1948.

\end{thebibliography}
\bibliographystyle{acm}

\end{document}